\documentclass[12pt]{amsart}
\usepackage{geometry} 
\usepackage{pgfplots}
\usepackage{psfrag}
\usepackage{amsmath}
\usepackage{calc}
\usepackage{systeme}
\usepackage{amsfonts}
\usepackage{amsthm}
\usepackage{algorithm2e}
\usepackage[applemac]{inputenc}
\usepackage{bbold}
\usepackage[numbers]{natbib}
\newtheorem{hyp}{Hypothesis}

\numberwithin{equation}{section}

\newtheorem{Theorem}{Theorem}
\newtheorem{Lemma}{Lemma}[section]
\newtheorem{Rem}{Remark}[section]

\newtheorem{Prop}{Proposition}[section]
\newtheorem{Cor}{Corollary}[section]
\newtheorem{Ex}{Example}[section]

\newcommand{\ca}{\color{black}}
\newcommand{\be}{\begin{equation}}
\newcommand{\ee}{\end{equation}}
\newcommand{\R}{\mathbb{R}}

\title{A singular infinite dimensional Hamilton-Jacobi-Bellman equation arising from a storage problem}
\author{Charles Bertucci$^1$, Jean-Michel Lasry$^2$, Pierre-Louis Lions$^{2,3}$}
\thanks{$^1$ : CMAP, Ecole Polytechnique, UMR 7641, 91120 Palaiseau, France\\
$^2$ :Universit\'e Paris-Dauphine, PSL Research University,UMR 7534, CEREMADE, 75016 Paris, France\\
$^3$ : Coll\`ege de France, 3 rue d'Ulm, 75005, Paris, France
}
\date{} 

\begin{document}

\maketitle

\begin{abstract}
In the first part of this paper, we derive an infinite dimensional partial differential equation which describes an economic equilibrium in a model of storage which includes an infinite number of non-atomic agents. This equation has the form of a mean field game master equation. The second part of the paper is devoted to the mathematical study of the Hamilton-Jacobi-Bellman equation from which the previous equation derives. This last equation is both singular and set on a Hilbert space and thus raises new mathematical difficulties.
\end{abstract}

\tableofcontents

\section{Introduction}
In this paper, we study a particular infinite dimensional Hamilton-Jacobi-Bellman (HJB in short) equation which arises in the modeling of an economic equilibrium problem. This problem, of a new type, arises from \ca the \color{black} interaction of a large number of facilities of storage between each other. The description of this model is the subject of the first part of this paper while the second one is concerned with a mathematical analysis of the HJB equation yielded by this modeling part.\\

The model which produces the equation of interest is concerned with the storage of a good in an infinite number of sites. At each site, an equilibrium takes place between supply, demand, storage and carriers who can bring the good from one site to the other.\\

Writing the equilibrium equations for the transfer of goods leads to an equation of the form of a Mean Field Game (MFG for short) master equation. We refer to \citep{lions2007cours,cardaliaguet2019master,bertucci2021monotone,bertucci2021monotone2} for more details on MFG master equations and to \citep{lasry2007mean,carmona2018probabilistic,carmona2018probabilistic2} for more on MFG. Let us insist on the fact that the MFG master equation arising from our modeling is not of the exact nature of most of the master equations studied in the literature because i) no precise game between the players is written, ii) the state variable is not the repartition of agents in the state space but rather the repartition of the product (or the good) in the state space. The fact that master equations can describe economic equilibrium outside of the usual MFG setting has already been remarked in \citep{bertucci2020mean,achdou2022class} and we believe it is a general feature of equilibrium models.

Since we are considering an infinite number of sites, the associated master equation is set on a Hilbert space. Quite remarkably, this equation derived from a HJB equation, which is thus naturally associated to the problem of a social planner. We focus our mathematical study on this singular HJB equation on a Hilbert space.\\

\ca
The study of HJB equations on Hilbert spaces dates back to early works from Barbu and Da Prato \citep{barbu1981global,barbu1981direct,barbu1983hamilton,barbu1985hamilton}. Those works rely on the convexity of the problem. An approach relying mostly on the regularizing properties of second order terms is presented in Da Prato and Zabczyk \citep{da2002second}, namely to deal with Lipschitz or convex Hamiltonians. Separately, the theory of viscosity solution was developed to treat non convex problems by Crandall and Lions \citep{crandall1985hamilton1,crandall1986hamilton2,crandall1986hamilton3}. Other important developments were made around the use of viscosity solution in infinite dimension, namely the study of the second order case by Lions \citep{lions1989viscosity} and the study of equations with some singular terms by Tataru \citep{tataru1992viscosity}. More recently, the textbook \citep{fabbri2017stochastic} by Fabbri, Gozzi and Swiech provides an overview of both recent developments involving viscosity solutions and the notion of so-called mild solutions, which is a way to define the solution of the HJB equation through stochastic representation formulas.

Let us also mention that in the last few years, the study of HJB equations in infinite dimensional space has attracted quite a lot of attention for the study of so-called potential MFG, see for instance \citep{lions2007cours,gangbo2015existence,cecchin2022weak}. Because of the particular nature of the singular HJB equation we are interested in, it is not of the same nature as most of the equations considered in this literature.\\

The equation at interest in this paper does not fall in the scope of the literature on infinite dimensional HJB equation, namely because of the presence of both a singular Hamiltonian and a full Laplacian in the Hilbert space. The mathematical analysis we provide to deal with these difficulties rely both on the theory of viscosity solutions and on the convexity of the solutions of the equation. Our strategy can be described simply by saying that we establish regularity estimates for finite dimensional approximations of the equation, which are uniform in the dimension, and then use the theory of viscosity solution to pass to the limit as the dimension goes to infinity. However, as we shall see, establishing the convergence relies on fine regularity properties of the solution because of the singularity of the problem\\

The rest of the paper is organized as follows. In a first time, we derive formally a singular infinite dimensional HJB equation from an Economic model of storage. In a second time, we study this HJB equation, namely by providing results of existence and uniqueness of an adequate notion of solution. \ca

\section{A toy model for storage in a large number of separated units}
\subsection{Description of the model}
There are $N \geq 2$ sites, one good and several populations of agents. On each site, there are local consumers and suppliers, local storage, and arbitrageurs. There are also carriers who transfer the good from site $n$ to site $n+1$, and vice versa ; moreover, $N+1$ is identified with $1$. \ca A good picture of these sites is that they are uniformly positioned on the unit circle.\ca\\

These agents interact through local markets, one market at each site. We denote by $p_{n,t}$ the price of the good on the site $n$ at time $t$, and by $k_{n,t}$ the level of storage of the good on the site $n$ at time $t$. In our model, we assume that $k_{n,t}$ can take any positive or negative value. In the real world, storage is bounded : the \ca lower \ca bound is often $0$ while the upper bound is due to physical or technical capacities. Moreover, in  many cases, for operational reasons, there is a high targeted level of storage. In these cases, our level of storage of $k_{n,t}$ would represent the difference between the actual level of storage and this high targeted level. Introducing lower and upper bounds for the storage, our model would become both more realistic, more interesting and much more difficult from a mathematical viewpoint as was highlighted in \citep{achdou2022class}. However the focus of the present paper is different, we want to show the effect that a multitude of sites of storage can have on the price.

Apart from shocks which are described below, the consumption flows are supposed to be given by demand functions $D_n(p_{n,t})$ and the supply flows are supposed to be given by supply functions $S_n(p_{n,t})$. \ca Let us insist on the fact that the previous demand function does not take into account the demand of arbitrageurs for the storage. \ca Additionally we assume that there are shocks on supply and demand so that the net supply, during a time interval of length $dt$, is given \ca for all time $t$ \ca by
\begin{equation}
(S_n(p_{n,t}) - D_n(p_{n,t}))dt + \sigma d W_{n,t}, 
\end{equation}
where $(W_n)_{n \geq 1}$ is a collection of $i.i.d.$ Brownian motions on a standard probability space.

\subsection{The equilibrium equations}
\ca As the only objective of this section is to derive an HJB equation (in particular we do not establish rigorous results), we make the two following assumptions.
\begin{itemize}
\item We assume that the state of this economy is simply given by the level of storage at each sites $K = (k_n)_{1\leq n\leq N}$. In particular, we assume that all the quantities are functions only of the state variable $K$. For instance we can write $p_{n,t} = p_n(K_t)$.
\item All the functions are smooth functions of $K$.
\end{itemize} 
\ca
We denote by $f_{n,t}$ the, algebraic, flow of transfer from $n$ to $n+1$ at time $t$. We assume that the global cost of transfer is given by $\frac c 2 f_{n,t}^2$. Transfers are assumed to be fast enough, so that at equilibrium, the marginal cost of transfer is equal to the difference of prices in sites $n$ and $n+1$, \ca i.e. for all time $t>0$\ca
\begin{equation}
cf_{n,t} = p_{n+1,t}-p_{n,t}.
\end{equation}

Arbitrageurs own the stored goods. The global cost of storage is given by a function $g_n(k_{n,t})$. Arbitrageurs are assumed to be risk neutral and to discount their future revenues at the rate $r \geq 0$. Hence, at equilibrium, the following holds \ca for all $t,dt > 0$\ca
\begin{equation}\label{noarbitrage}
p_{n,t} = \mathbb{E}[p_{n,t+dt}]e^{-rt} - g'_n(k_{n,t}),
\end{equation}
where the expectations is taken with respect to the collection of independent Brownian motions. Indeed, \ca remark that the left hand side of \eqref{noarbitrage} is simply the price at the site $n$ at the time $t$ of a unit of good, while the right hand side is the expected price at the same site $n$, at the time $t+dt$ minus the cost of storing a unit of good. Indeed, since $g_n$ represents the total cost of storage at the site $n$, storing a single unit of good when the storage level is $k_{n,t}$ is given by $g'_n(k_{n,t})$. Hence the no arbitrage relation indeed yields \eqref{noarbitrage}.\ca Note that the previous relation indeed holds because there is no constraint on the storage. If there were state constraints on the storage, then the situation would be closer to the one studied in \citep{achdou2022class}, and only an inequality would hold when $k_{n,t}$ reaches a constraint.\\

\ca Thanks to the assumption that $p_{n,t}$ is only a function of the state variable $K$, we can rewrite \eqref{noarbitrage} into
\be\label{eq133}
p_n(K_t) = \mathbb{E}[p_n(K_t)]e^{-rt} - g'_n(K_t),
\ee
where we use the slight abuse of notation that $g'_n(K_t) = g'_n(k_{n,t})$. Now let us remark that the dynamics of $(K_t)_{t \geq 0}$ are given by
\be
dk_{n,t} = (S_n(p_{n,t}) - D_n(p_{n,t}))dt + \sigma d W_{n,t} + (cf_{n,t} - c f_{n-1,t})dt.
\ee
Hence, we can deduce from \eqref{eq133} the relation
\begin{equation}\label{mfgp}
\begin{aligned}
0 = -&r p_n(K) + \frac{\sigma^2}{2}\sum_{i=1}^N\partial_{k_ik_i}p_n(K) - g'_n(k_n) +\\
+ & \sum_{i =1}^N \partial_{k_i}p_n(K)\left[S_n(p_{n}) - D_n(p_{n}) + \frac{p_{n+1}+ p_{n-1} - 2p_n}{c}\right],
\end{aligned}
\end{equation}
which has to be valid for any $K \in \R^N$. \ca\\

Let us remark that equation \eqref{mfgp} has the typical form of a \ca stationary \ca MFG master equation \ca in finite state space\ca  \citep{bertucci2019some,bertucci2021monotone}. However, contrary to the usual MFG master equations, the variable $K$ does not describe here the repartition of players in the state space but rather the repartition of objects in the state space (here the level of storage in each site). \ca It is a general feature of MFG master equation as they naturally appear as the equilibrium equation of several Economics model, even if no proper game takes between the agents. For instance, this is also the case in \citep{bertucci2020mean,achdou2022class}.
\ca

Finally, since there is no friction in our market equilibrium model, not surprisingly, invisible hand principle applies and the MFG master equation \eqref{mfgp} can be derived from a HJB equation which is the HJB equation of a benevolent planner of the invisible hand. We come back on this later on.

\subsection{The limit equations}
\ca In this section, we keep up with the assumptions of the previous one, namely on the smoothness of the functions $(p_n)_{1\leq n \leq N}$. Moreover, as we are going to present the limit of the equation \eqref{mfgp} as $N \to \infty$, we also assume here that all the convergence we are going to show hold. \ca

We now compute the limit of \eqref{mfgp} as $N \to \infty$. \ca Recall that we can interpret the $N$ sites as being positioned uniformly on the unit circle. Taking the limit $N \to \infty$, we expect that there will be a continuum of sites located at each $x \in [0,1]$ (in a periodic setting).

In the previous equations, there are two variables. The first one is the location of the site in question $n$. As we just mentioned, in the limit $N \to \infty$, this variable is replaced by a continuous one $x \in [0,1]$. The second variable is the level of storage in each of the location $K\in \mathbb{R}^N$. $K$ can be seen as a function from the set of sites to $\R$. We expect that taking $N \to \infty$, this variable $K$ is a replaced by a function $k : [0,1] \to \mathbb{R}$. 

Hence, $p_n(K)$ becomes in the limit $p(x,k)$. Furthermore, $p$ can be seen as an operator associating to $k \in L^2([0,1])$ the function $x \to p(x,k)$. We assume here that $p$ is well defined as an object $p : L^2([0,1]) \to L^2([0,1])$.\ca\\

To simplify the following, we assume \ca that $S_n(p) - D_n(p) = p$ for all $n$. The proper scaling for which a passage to the limit $N \to \infty$ is meaningful is $c = N^{-2}$\ca. Since we assume that the sites are located uniformly on a unit circle, this corresponds to a quadratic cost of transportation. In this regime, the limit of \eqref{mfgp} becomes
\begin{equation}\label{eqpl}
0 = - r p(k) + \frac{\sigma^2}{2}\Delta_k p(k) + \langle p(k) -  \Delta_x p(k), \nabla\rangle p(k) + g'(k), \text{ in } L^2([0,1]).
\end{equation}
In the previous equation, the operator $\Delta_k$ is the Laplacian operator in the space of functions taking values in $L^2([0,1])$, i.e., it is the trace of the Hessian; $\langle\cdot,\cdot\rangle$ is the scalar product in $L^2([0,1])$, $\Delta_x$ is the usual Laplacian operator from $H^2([0,1])$ into $L^2([0,1])$ and $\nabla$ is the usual gradient operator on the Hilbert space $L^2([0,1])$. \ca Of course, it is by no mean clear that classical solutions of \eqref{eqpl} exist. Furthermore, as \eqref{mfgp} was a finite state space MFG master equation, \eqref{eqpl} is a MFG master equation set on a Hilbert space. At the moment, there exists no technique to study this master equation directly. We now explain when this master equation can be obtained from a control problem, or in other words when it derives from an HJB equation or also when we are in the potential case. \\

\subsection{The HJB equation in the potential case}
Using the usual terminology of game theory, we say that the game derives from a potential when there exists a differentiable function $G : L^2([0,1]) \to \R$ such that $\nabla G (k)= g(k)$. Assume that such a $G$ exists and assume that there exists a unique solution $\phi : L^2([0,1])\to \mathbb{R}$ of the equation
\begin{equation}\label{eqphi1}
r \phi -  \frac{\sigma^2}{2}\Delta_k \phi - \frac{1}{2} \langle (Id - \Delta_x)\nabla \phi,\nabla \phi\rangle = G(k) \text{ in } L^2([0,1]),
\end{equation}
Then, $\nabla \phi$ is a solution of \eqref{eqpl}. Hence, solving \eqref{eqpl} can be done by means of the HJB equation \eqref{eqphi1}.\\

The study of \eqref{eqphi1} is easier than the one of \eqref{eqpl}. The interest of the latter, is that it offers a wider range of applications. Indeed, some models can be expressed by equation of the form of \eqref{eqpl} without the possibility of integrating the equation into an HJB equation such as \eqref{eqphi1}. This is for instance the case for general storage costs $g$. However, for the sake of simplicity, we limit ourselves to the study of a time dependent counterpart of \eqref{eqphi1} in this paper.\\

The rest of the paper is devoted to the mathematical study of HJB equation of the form of \eqref{eqphi1}.

\ca

\section{Notation and mathematical formulation of the problem}
Let $(H, \langle \cdot, \cdot \rangle)$ be a separable real Hilbert space and $A \in \mathcal{L}(H)$ a symmetric, positive and invertible operator such that $A^{-1}$ is compact. We denote by $(\lambda_i)_{i \geq 0}$ the increasing sequence of eigenvalues of $A$ (possibly with repetition according to the multiplicity) and by $(e_i)_{i \geq 0}$ an orthonormal basis of $H$ formed of corresponding eigenvectors. \ca We assume at all time that $A$ is such that 
\be\label{hypA}
\sum_{ i \geq 0}\frac{\log(1+ \lambda_i)}{\lambda_i} < \infty.
\ee
\ca We shall adopt the convention 
\begin{equation}
\forall x \in H, x = \sum_{i=0}^{\infty} x_i e_i, \text{ where } (x_i)_{i \geq 0} \in \mathbb{R}^{\mathbb{N}}.
\end{equation}
\ca Moreover we introduce the notation $H_N = Span(\{e_0,...,e_N\})$.\ca

We consider the following PDE whose solution is denoted by $\phi : (0,\infty)\times H \to \mathbb{R}$ 
\begin{equation}\label{hjb}
\partial_t \phi - \Delta \phi + \frac{1}{2}\langle A \nabla \phi, \nabla \phi \rangle = 0, \text{ in } (0,\infty) \times H,
\end{equation}
\begin{equation}\label{hjbt0}
\phi|_{t = 0} = \phi_0 \text{ in } H.
\end{equation}
In the previous equations, $\phi_0$ is considered as a data from the model and assumptions on it shall be made later on, $\nabla$ is the gradient operator and \ca $\Delta$ is the Laplacian, which is defined by the trace of the Hessian operator. Let us note that the Laplacian is not always well defined, even on smooth functions, as there is no reason that the Hessian is in the trace class. \ca The aim of this paper is to provide a suitable notion of solution for the singular PDE \eqref{hjb}-\eqref{hjbt0}.
\begin{Rem}
The case of equation \eqref{hjb} with a right hand side term could have also been treated. Although, because it is merely an extension of the study we present here, we leave this trivial extension to the interested reader.
\end{Rem}

\subsection{The quadratic case}
Let us consider first the instructive case in which $\phi_0$ is given by
\begin{equation}
\phi_0(x) = \frac{1}{2}\sum_{i =0 }^{\infty} \mu^0_i x_i^2,
\end{equation}
for some bounded real sequence $(µ^0_i)_{i \geq 0} \in \ell^{\infty}$. We then seek a solution of \eqref{hjb}-\eqref{hjbt0} of the form
\begin{equation}\label{phiclosedform}
\phi(t,x) =  c(t) + \frac{1}{2}\sum_{i = 1}^{\infty} µ_i(t)x_i^2, \text{ for } t \geq 0, x \in H.
\end{equation}
When plugging this function in \eqref{hjb}-\eqref{hjbt0}, we obtain that $\phi$ is indeed a solution if and only if for all $i \geq 0$
\begin{equation}\label{edomu}
\begin{cases}
\frac{1}{2} \dot{µ}_i(t) + \frac{1}{2}\lambda_i µ_i(t)^2 = 0,\\
µ_i(0) = µ_i^0,
\end{cases}
\end{equation}
and 
\begin{equation}
\dot{c}(t) = \sum_{i =0}^{\infty} µ_i(t).
\end{equation}
From this we deduce that if for some $i \geq 0$, $µ_i^0 < 0$, then there is no solution to \eqref{hjb}-\eqref{hjbt0} as there is explosion in finite time. On the other hand, if $\mu_i^0 \geq 0$ for all $i \geq 0$, then the system \eqref{edomu} leads to 

\begin{equation}\label{eq:241}
\forall i \geq 0, t\geq 0, µ_i(t) = \frac{µ_i^0}{1 + \lambda_i µ_i^0 t}.
\end{equation}
This naturally yields a formula for $c$ which is
\begin{equation}\label{quadrac}
\forall t \geq 0, c(t) = \sum_{i =0}^{\infty} \lambda_i^{-1} \log(1 + \lambda_i µ_i^0 t).
\end{equation}
The previous is well defined only if $\sum_i \lambda_i^{-1}\log(1 + \lambda_i)< + \infty$. The previous computations can be summarized in
\begin{Prop}
For any positive bounded sequence $µ^0$, the function $\phi$ defined by \eqref{phiclosedform},\eqref{eq:241} and \eqref{quadrac} is a classical solution of \eqref{hjb}, which satisfies the appropriate initial condition.
\end{Prop}
The question of uniqueness of such a solution is treated later on. Let us insist upon the fact that, in this simple setting, the sequence $(\lambda_i)_{i \geq 0}$ has to satisfy 
\begin{equation}
\sum_i \lambda_i^{-1}\log(1 + \lambda_i)< + \infty,
\end{equation}
and $\phi_0$ has to be assumed to be convex. \ca Indeed, in order to get that the Laplacian of $\phi$ is well defined for $t > 0$, we only need that $\phi_0$ is convex and the assumption $\sum_i \lambda_i^{-1} < \infty$. However, in ordre to define properly $\phi$ with the function $c(\cdot)$, we need in addition that $\phi_0 \in \mathcal{C}^{1,1}$ and \eqref{hypA}\ca.

\subsection{The deterministic case}
Let us consider the deterministic problem
\begin{equation}\label{deterministic}
\begin{aligned}
\partial_t \psi + \frac{1}{2}\langle A \nabla \psi, \nabla \psi\rangle &= 0, \text{ in } (0,\infty)\times H,\\
\psi|_{t = 0} &= \phi_0 \text{ in H}.
\end{aligned}
\end{equation}
Formally the solution of this equation is given by the Lax-Oleinik formula
\begin{equation}\label{laxoleinik}
\psi(t,x) = \inf_{y \in H} \left\{ \phi_0(y) + \frac{\langle A^{-1}(x-y), x-y \rangle}{2t}\right\}.
\end{equation}
Of course it is not clear that this formula \ca defines a smooth functions $\psi$ \ca without assumptions on $A$ and $\phi_0$. Nonetheless, we believe that the deterministic case is insightful for the study of \eqref{hjb}. For instance, it helps us to understand the behavior of the solution of \eqref{hjb} near $t = 0$. Indeed the function $\psi$ defined by \eqref{laxoleinik} is not continuous in $t$ at $t=0$ without additional assumptions on the initial condition. The following result gives a necessary condition for which this is the case.
\begin{Prop}\label{prop:psi0}
Assume that $\phi_0$ is convex, continuous and that $\phi_0(x) \to \infty$ when $|x| \to \infty$. Then for any $x \in H$, $\psi(t,x) \to \phi_0(x)$ as $t \to 0$ where $\psi$ is defined by \eqref{laxoleinik}.
\end{Prop}
\begin{proof}
Let us observe first that $\psi(t,x) \leq \phi_0(x)$ for all $t,x$ (simply choose $y = x$ in \eqref{laxoleinik}). Fix $x \in H$ and consider a sequence $(t_n)_{n \geq 0}$ of non negative elements which converges toward $0$. Consider for all $n \geq 0$ an element $y_n \in H$ such that
\begin{equation}\label{minilax}
\phi_0(y_n) + \frac{\langle A^{-1}(x-y_n), x-y_n \rangle}{2t_n} \leq \psi(t_n,x) + \frac{1}{n+1} \leq \phi_0(x) + \frac{1}{n+1}.
\end{equation}
From the growth assumption on $\phi_0$, we deduce that, up to a subsequence, $(y_n)_{n \geq  0}$ has a weak limit $y^* \in H$. From \eqref{minilax}, since $\phi_0$ is clearly bounded from below, we deduce that \ca $(\langle A^{-1}(x-y_n),x-y_n\rangle)_{n \geq 0}$ converges toward $0$. Hence we obtain that $y^* = x$. \ca Since $\phi_0$ is lower semi-continuous for the weak topology, we deduce finally the desired result.
\end{proof}
\ca \begin{Rem}
It is in fact only the coercivity and the lower semi-continuity of $\phi_0$ with respect to the weak topology which are needed.
\end{Rem} \ca
We also gather the following properties on the function $\psi$ defined by \eqref{laxoleinik}.
\begin{Prop}
Assume that $\phi_0$ is convex, continuous and that $\phi_0(x) \to \infty$ when $|x| \to \infty$. Then $\psi$ is globally continuous, and there exists a unique $y^*\in H$ such that
\be
\psi(t,x) = \phi_0(y^*) + \frac{\langle A(x-y^*),x-y^*\rangle}{2t}.
\ee
Moreover, for all $0 < t \leq s, x \in H$
\be
\psi(t,x) \geq \psi(s,x),
\ee 
\be\label{boundeter}
\psi(t,x) \leq \min\left\{\frac{\langle Ax,x\rangle}{2t} + \phi(0),\phi_0(x)\right\}.
\ee
\end{Prop}
\begin{proof}
The two inequalities are almost trivial, the first one is simply a comparison element by element in the infimum, while the second one comes from taking $y = 0$. 

The existence of a unique minimizers holds because, for all $t >0, x \in H$, $y \to \phi_0(y) + (2t)^{-1}\langle A(x-y),x-y\rangle$ is continuous, coercive and strictly convex. The continuity can be proven in exactly the same fashion as in the previous Proposition so we do not detail it here. 
\end{proof}

\begin{Rem}\label{rem:uniformpsi}
Since $\psi$ is continuous in $x$ for all $t > 0$ and for all $0 < t  \leq s, x \in H$ $\psi(t,x) > \psi(s,x)$, we deduce from Dini's Theorem that, in fact, the convergence in Proposition \ref{prop:psi0} is uniform on all compact sets.
\end{Rem}
\ca

Another situation in which the deterministic case is helpful is the one in which we have estimates on the Laplacian of a solution $\phi$ of \eqref{hjb}. In this case, we can use the solution given by \eqref{laxoleinik} (if it is well defined) to obtain some continuity estimates near $t=0$ on $\phi$ as the next result explains.
\begin{Prop}\label{prop:est0}
Let us consider a smooth solution $\phi$ of \eqref{hjb} such that $\phi_{|t = 0}$ is continuous, convex and goes to $\infty$ as $|x| \to \infty$ and such that $\Delta \phi$ is uniformly bounded, i.e. $\|\Delta \phi\|_{\infty} \leq c$. Moreover consider the function $\psi$ defined by \eqref{laxoleinik} with $\phi_0 = \phi|_{t = 0}$ and assume that it is a viscosity solution of \eqref{deterministic}. Assume moreover that 
\begin{equation}\label{hyp:o}
|\phi(t,x) - \psi(t,x)| \underset{|x|\to \infty}{=} o(\psi(t,x)).
\end{equation}
Then the following holds for $t \geq 0, x \in H$
\begin{equation}
\psi(t,x) - ct \leq \phi(t,x) \leq \psi(t,x) + ct.
\end{equation}
\end{Prop}
Let us comment on \eqref{hyp:o}. This assumption controls the distance between $\psi$ and $\phi$. For instance, if the difference of these two functions is bounded, then \eqref{hyp:o} holds. Indeed recall that $\psi$ is a coercive function of $x$ for any $t$. Moreover, let us insist on the fact that the bounds obtained in the previous Proposition do not depend quantitatively on this assumption, but only on the fact that it holds. This shall be helpful later on in the paper.
\begin{Rem}
Let us remark that if $\psi$ is convex, then one can replace the conclusion of the Proposition by, for all $t\geq 0, x \in H$ :
\begin{equation}
\psi(t,x) \leq \phi(t,x) \leq \psi(t,x) + ct.
\end{equation}
\end{Rem}
\begin{Rem}\label{Rem:l1}
If instead of the uniform bound on the Laplacian, one has indeed $\|\Delta \phi(t,\cdot)\|_{\infty} \leq c(t)$ for some integrable function $c(\cdot)$. Then one can replace the conclusion of the Proposition with
\begin{equation}
\psi(t,x) -\omega(t)\leq \phi(t,x) \leq \psi(t,x) + \omega(t),
\end{equation}
where $\omega(t) = \int_0^tc(s)ds$.
\end{Rem}
\begin{proof}\ca\textit{(Of Proposition \ref{prop:est0}.)}\ca
We only prove the second inequality, the first one can be obtained in a similar fashion. The following is by now somehow standard in the study of Hamilton-Jacobi equations in infinite dimension.\\
Without loss of generality we can assume that $\phi |_{t = 0} \geq 0$. Indeed, the equation is invariant by the addition of a constant and $\phi |_{t = 0}$ is bounded from below. We want to prove that for any $\lambda \in (0,1)$, $t \geq 0, x \in H$,
\begin{equation}\label{eq:34}
\lambda \phi(t,x) \leq \psi(t,x) + ct.
\end{equation}
Let us assume that \eqref{eq:34} does not hold. Hence there exists $(t_*,x_*) \in (0,\infty) \times H$ such that $\lambda \phi(t_*,x_*) \ca > \ca \psi(t_*,x_*) + ct_*$. Thus there exists $\delta > 0$, such that $w : (t,x) \to \psi(t,x) + ct - \lambda \phi(t,x) + \delta t$ satisfies $w|_{t = 0} \geq 0$ and $w(t_*,x_*)< 0$. From this, we deduce that $t_0$ defined by
\begin{equation}
t_0 = \inf\{t > 0, \exists x \in H, w(t,x) < 0\}
\end{equation}
is smaller than $t_*$. From the coercivity of $\psi(t_0)$ and \eqref{hyp:o}, we deduce that $w(t_0)$ is also coercive. Hence, there exists $R > 0$ such that 
\begin{equation}
\inf_{|x| \leq R} w(t_0,x) = 0.
\end{equation}
\ca Indeed, by construction of $t_0$ and from the continuity in time of $w$, we deduce that $\inf_x w(t_0,x) = 0$, and because $w$ is coercive, that the infimum is located in a bounded set. \ca

From Stegall's Lemma \citep{stegall1978optimization,stegall1986optimization}, which is always true in Hilbert spaces, we deduce that for any $\epsilon > 0$, there exists $\xi \in H$, $|A\xi| \leq \epsilon$ and $x_0 \in H$ such that $x_0$ is a strict minimum of $ x \to w(t_0,x) + \langle \xi,x\rangle$ on $\{x \in H, |x|\leq R\}$. By construction of $t_0$ and $x_0$, \ca since $\psi$ is a vsicosity subsolution of the deterministic equation, we obtain that

\begin{equation}
\partial_t(\lambda \phi(t_0,x_0)) - \delta - c +\frac 12 \langle A \nabla\lambda \phi(t_0,x_0) - \xi,\nabla\lambda\phi(t_0,x_0)- \xi\rangle \geq 0.
\end{equation}
From the assumption we made on $\phi$, we can use the PDE it is a solution of to obtain that \ca
\begin{equation}
\lambda\Delta \phi(t_0,x_0) - \delta - c + (\lambda^2 - \lambda)\frac12 \langle A \nabla \phi(t_0,x_0), \nabla \phi(t_0,x_0)\rangle + \lambda \langle A \nabla \phi(t_0,x_0), \xi\rangle + \frac 12 \langle A \xi,\xi\rangle \geq 0.
\end{equation}
Recall now that $A$ is a positive operator, hence $\langle A \nabla \phi(t_0,x_0), \nabla \phi(t_0,x_0)\rangle \geq 0$. Thus we obtain that, taking $\epsilon$ small enough ($\epsilon$ was used to measure the size of $\xi$), we obtain a contradiction in the previous inequality. Hence for any $\lambda \in (0,1)$, \eqref{eq:34} holds. The required result is then proved by passing to the limit $\lambda \to 1$ in \eqref{eq:34}.

\end{proof}

\ca
\begin{Rem}
The sense in which $\psi$ is assumed to be a viscosity solution of the singular HJB equation \eqref{deterministic} could have been made precise, namely because of the singularity. However, because, we are only using this property on the test function $\phi$ which is such that $\langle A \nabla \phi,\nabla \phi\rangle$ is well defined, this does not raise any particular difficulty as any reasonable notion of viscosity solution yields the same result. Let us insist that $\langle A \nabla \phi,\nabla \phi\rangle$ is indeed well defined since $\phi$ is assumed smooth and $\Delta \phi$ bounded.
\end{Rem}
\ca

\subsection{Counterexamples to the time continuity of the solution of the deterministic problem}

We present two examples in which $\phi_0$ does not satisfy the assumptions of the previous result to justify in some sense that some assumptions have to be made.\\

The first one is an example in which $\phi_0$ does not satisfy the required growth assumption at infinity. 
\begin{Ex}
Assume that the eigenvalues of $A$ are given by $\lambda_n = (n+1)^{\alpha}, n \geq 0$ and consider $x^* \in H$ given by $x^*_n = (n+1)^{-\beta}, n \geq 0$. Let us choose $\phi_0(x) = \langle x,x^*\rangle^2$. We take $\alpha$ and $\beta$ such that $\alpha > 2 \beta > 1$. The following holds for any $x \in H$
\begin{equation}\label{resultex}
\inf_{y \in H} \left\{ \phi_0(y) + \frac{1}{2t}\langle A^{-1}(x-y), x-y \rangle\right\} \underset{t \to 0}{\longrightarrow} 0.
\end{equation}
Indeed fix $x \in H$ and consider $z \in H$ defined by
\begin{equation}
\begin{cases}
z_k = x_k, k \leq N-1\\
z_N = -(x^*_N)^{-1} \sum_{i = 0}^{N-1}x^*_iz_i,\\
z_k = 0, k \geq N+1
\end{cases}
\end{equation}
where $N$ is to be fixed later on. We compute
\begin{equation}
\begin{aligned}
\phi_0(z) + \frac{1}{2t}\langle A^{-1}(x-z), x-z \rangle &= \frac{1}{2t} \left( \lambda^{-1}_N (z_N - x_N)^2 + \sum_{i \geq N+1} \lambda^{-1}_i x_i^2\right),\\
& \leq \frac{1}{2t} \left( \lambda^{-1}_N (z_N - x_N)^2 + \sum_{i \geq N+1} \lambda^{-1}_i x_i^2\right),\\
& \leq \frac{C}{2t} \left( (N+1)^{-\alpha} (N+1)^{2\beta} + ( N+1)^{1 - \alpha} \right).
\end{aligned}
\end{equation}
where $C$ is a constant depending on $\alpha, \beta$ and $x$. Then, since $\alpha > 2 \beta > 1$, choosing $N$ such that $(N+1)^{2\beta - \alpha}t^{-1}$ is as small as we want, we conclude that \ca the infimum in \eqref{resultex} is less than or equal to $0$, since $\phi_0 \geq 0$ we obtain that \eqref{resultex} \ca  indeed holds and thus that the initial condition is not satisfied in such a case. \end{Ex}

The second example is a case in which $\phi_0$ is not convex. This non-convexity highlights the non weak lower semi continuity of $\phi_0$ which is crucial and which we also comment in the following example.
\begin{Ex}
Consider $\phi_0$ defined by
\begin{equation}
\phi_0(x) = \begin{cases} |x|^2 \text{ if } |x| \geq 1,\\ 2 - |x|^2 \text{ else.} \end{cases}
\end{equation}
For any $t > 0$, one has $\psi(t,0) = 1$. Indeed, one necessarily has $\psi(t,x) \geq \inf_{x \in H} \phi_0(x) = 1$. Then considering the sequence $y_n = e_n$ as a minimizing sequence in \eqref{laxoleinik}, it follows that $\psi(t,0) \leq 1$, hence the fact that $\psi(t,0) = 1$. This is an obvious contradiction to the fact that $\psi$ converges toward $\phi_0$ as $t \to 0$. 

More generally, it is the weak lower semi continuity of $\phi_0$ which is important to obtain the time continuity in $t= 0$ of the deterministic solution. More precisely, let us assume that there is a sequence $(x_n)_{n \geq 0}$ weakly converging toward $x \in H$ such that both $\liminf \phi_0(x_n) < \phi_0(x)$ and $(A^{-1}x_n)_{n \geq 0}$ converge (strongly) toward $A^{-1}x$. Then for any $t\geq 0$, $\psi(t,x) = \liminf \phi_0(x_n) < \phi_0(x)$.
\end{Ex}

\subsection{A priori estimates}\label{sec:apriori}
In this section we present some key a priori estimates to establish existence and uniqueness of solutions of  \eqref{hjb}. Since we are not going to use those estimates directly, \ca but rather their finite dimensional counterpart, \ca we prove them for a very specific class of functions. Nevertheless, those estimates are insightful on why the equation \eqref{hjb} is well posed.

For a smooth function $\phi : H \to \mathbb{R}$ and $i,j \geq 0$, we shall denote by $\phi_i$ the derivative of $\phi$ in the direction of $e_i$ and by $\phi_{ij}$ the derivative of $\phi_i$ in the direction of $e_j$. 

We define the space $\mathcal{B}$ of functions $[0,\infty)\times H \to \mathbb{R}$ such that
\begin{itemize}
\item $\phi \in \mathcal{C}^4$.
\item For all $i,j \geq 0,  \Delta \phi_{i}, \Delta \phi_{ij}$ are well defined.
\item For all $i \geq 0$, $t \geq 0$, $\sum_{k \geq N}\partial_{kk} \phi_{ii}(t,x) \to 0$ as $N \to \infty$, uniformly in $x \in H$.
\item $\nabla \phi, \nabla \phi_i, \nabla\phi_{ij} \in Dom(A)$.
\item $\forall t \geq 0, \|\phi_{ii}(t,\cdot)\|_{\infty} < \infty$.
\end{itemize} 
We acknowledge that this space $\mathcal{B}$ can seem a bit arbitrary at first. It is a suitable set in which the following a priori estimates are not too difficult to prove. \ca It satisfies the important property that any smooth function on the finite dimensional space $H_N$ can be lifted into $\mathcal{B}$. In particular, functions $\phi \in \mathcal{B}$ which are solutions of \eqref{hjb} satisfy the assumption of Proposition \ref{prop:est0}. \ca \\

As we already mentioned, the regularity of solutions of \eqref{hjb} can be established with a key a priori estimate on the second order derivatives of the solution that we now present.
\begin{Prop}\label{prop:estapriori}
If $\phi \in \mathcal{B}$ is a solution of \eqref{hjb}, then for all $t> 0$, $i \geq 0$, $x\in H$, 
\begin{equation}\label{estii}
\phi_{ii} (t,x) \leq \frac{\phi_{ii}^0}{1 + \lambda_i \phi_{ii}^0 t},
\end{equation}
where $\phi_{ii}^0 := \|\phi_{ii |t = 0}\|_{\infty}$. \ca It also holds that for any $t > 0, x, \xi \in H$,
\be\label{estD2}
\xi D^2\phi(t,x)\xi \leq t^{-1}\langle A^{-1}\xi,\xi\rangle.
\ee
\ca
Moreover, if $\phi_{|t = 0}$ is convex, then $\phi(t)$ is convex for all time. If $\phi_{|t = 0}$ is bounded from below by a constant, then $\phi$ is bounded from below by the same constant. \end{Prop}
\begin{Rem}
\ca Note that the previous result is a proof that the equation regularizes the initial condition. \ca
\end{Rem}
\begin{proof}\textit{(Of Proposition \ref{prop:estapriori}.)}
First let us remark that constants are solution of \eqref{hjb}, thus if $\phi_{|t = 0}$ is bounded from below by a constant, then $\phi$ is bounded from below by the same constant, using a comparison principle type result. This can be proved in exactly the same fashion as Proposition \ref{prop:est0} and thus we do reproduce the argument here for the sake of clarity.\\

For all $i,j \geq 0$, $\phi_{ij}$ is a solution of
\begin{equation}\label{eqphiij}
\partial_t \phi_{ij} - \Delta \phi_{ij} + \langle A\nabla \phi, \nabla \phi_{ij} \rangle + \sum_k \lambda_k \phi_{ki}\phi_{kj} = 0 \text{ in } (0,\infty)\times H. 
\end{equation}
Choosing $i = j$ in the previous equation yields
\begin{equation}\label{eqphiii}
\partial_t \phi_{ii} - \Delta \phi_{ii} + \langle A \nabla \phi,\nabla \phi_{ii}\rangle + \lambda_i \phi_{ii}^2\leq 0 \text{ in } (0,\infty) \times H.
\end{equation}
Assume now that \eqref{estii} does not hold and thus that there exists, $t_*,x_*$ such that 
\begin{equation}
\phi_{ii}(t_*,x_*) >\frac{\phi_{ii}^0}{1 + \lambda_i \phi_{ii}^0 t_*}.
\end{equation}
Let us define $u(t,x) = \phi_{ii}(t,x) - \frac{\phi_{ii}^0}{1 + \lambda_i \phi_{ii}^0 t} - \delta t - \epsilon |x|^2$. By construction of $(t_*,x_*)$, there exists $\delta, \epsilon > 0$ such that $u(t_*,x_*) > 0$. We consider such a pair $(\delta,\epsilon)$. Let us now define, for some $N \geq 1$, $t_0$ by
\begin{equation}
t_0 = \inf \{t > 0, \exists x \in H_N, u(t,(x,0)) > 0\}.\footnote{Let us recall that $(x,0)$ denotes the element of $H$ whose $N$ first components are equal to $x$ and the rest are set to $0$.}
\end{equation}
\ca For $N \geq 1$ sufficiently large, the set on which the infimum is taken is non-empty. Indeed, from the continuity of $u$, we know that $u(t_*,proj_{H_N}x_*) \to u(t_*,x_*)$ as $N \to \infty$. \ca It is clear that $t_0$ depends on $N$, however let us recall that it is bounded in $N$ since $u$ is continuous in $H$. Since $\phi_{ii}$ is bounded (indeed $\phi \in \mathcal{B}$), we know that there exists $R >0$ such that 
\begin{equation}
\sup \{u(t_0,(x,0)), x \in H_N, |x| \leq R\} = 0.
\end{equation}
Let us insist on the fact that $R> 0$ can be chosen independently of $N$ here. Hence, using once again Stegall's Lemma \citep{stegall1978optimization,stegall1986optimization}, we know that for any $\epsilon' > 0$ (independent of $N$), there exists $\xi \in H_N$, $|A\xi| \leq \epsilon'$ such that $x \to u(t,x) + \langle \xi,x\rangle$ has a strict maximum at some point $x_0 \in H_N$. Hence at we deduce the following
\begin{equation}
\begin{cases}
u(t_0,x_0) = O(\epsilon'),\\
\partial_t \phi_{ii}(t_0,(x_0,0)) \geq \delta - \lambda_i\left(\frac{\phi_{ii}^0}{1 + \lambda_i \phi_{ii}^0t_0}\right)^2,\\
\nabla_N \phi_{ii}(t_0,(x_0,0)) = 2\epsilon x_0 - \xi,\\
\sum_{k = 1}^N \partial_{kk}\phi_{ii} \leq 2 \epsilon N.
\end{cases}
\end{equation}
Evaluating \eqref{eqphiii} at $(t_0,(x_0,0))$, we obtain
\begin{equation}
\delta - \lambda_i\left(\frac{\phi_{ii}^0}{1 + \lambda_i \phi_{ii}^0}\right)^2 -2\epsilon N - \sum_{k \geq N+1} \partial_{kk}\phi_{ii} + \langle A \nabla_N \phi,2 \epsilon x_0- \xi\rangle + \lambda_i \phi_{ii}^2(t_0,(x_0,0)) \leq 0.
\end{equation}
Using the first line of \eqref{eqphiii}, we deduce
\begin{equation}
\delta - 2 \epsilon N - \sum_{k \geq N+1} \partial_{kk} \phi_{ii} + O(\epsilon |x_0|) + O(\epsilon') \leq 0.
\end{equation}
Let us now remark that, \ca from the construction of $(t_0,x_0)$, $\epsilon|x_0|^2$ remains bounded as $\epsilon \to 0$. Hence, we deduce that $\epsilon|x_0| \to 0$ as $\epsilon \to 0$. \ca  Moreover, using the uniform summability of $\Delta \phi_{ii}$ \ca($\phi \in \mathcal{B})$ \ca, we take $N$ large enough, then we take $\epsilon \to 0$ and then we take $\epsilon' \to 0$ to arrive at a contradiction in the previous equation. \ca Thus the claim on the upper bound on $\phi_{ii}$ follows.\\

We now prove the propagation of the convexity and the estimate on the Hessian. In order to do so, we introduce the function $w : \mathbb{R}_+ \times H^2 \to \mathbb{R}$ defined by
\be
w(t,x,\xi) = \sum_{i,j\geq 0} \phi_{ij}\xi_i \xi_j,
\ee
From \eqref{eqphiij}, it follows that $w$ is a solution of
\be\label{eqxxi}
\partial_t w - \Delta_x w + \langle A \nabla \phi,\nabla_x w\rangle + \frac 14 \langle A \nabla_\xi w,\nabla_\xi w\rangle = 0\text{ in } (0,\infty)\times H^2,
\ee
where $\nabla_x w$ and $\Delta_x w$ stands for respectively the gradient and the Laplacian of $w$ with respect to $x$ while $\nabla_\xi w$ is its gradient with respect to $\xi$. Once again, since $w$ is a classical solution of the previous equation (because $\phi \in \mathcal{B}$), by using a comparison result between $w$ and the constant $0$, we deduce that $w \geq 0$ if it is the case at the initial time, which implies the propagation of the convexity of $\phi(t,\cdot)$ for all $t \geq 0$.

\ca To prove the estimate on the Hessian, it suffices to remark that $(t,x,\xi) \to t^{-1}\langle A^{-1}\xi,\xi\rangle$ is a solution of \eqref{eqxxi}. Hence the result follows once again from a comparison principle.\ca

\end{proof}

We then deduce the following.
\begin{Cor}
If $\phi \in \mathcal{B}$ is a smooth solution of \eqref{hjb} with convex initial data, then for any $t > 0$, $i,j \geq 0$
\begin{equation}
|\phi_{ij}(t)| \leq \sqrt{\frac{\phi_{ii}^0}{1 + \lambda_i \phi_{ii}^0 t}}\sqrt{\frac{\phi_{jj}^0}{1 + \lambda_j \phi_{jj}^0 t}}
\end{equation}
\end{Cor}
\ca \begin{proof}
This Corollary simply follows from the non-negativity of the Hessian matrix of $\phi$ as well as form the bounds on $\phi_{ii}$. Indeed, take $i,j \geq 0$ and evaluate the non-negativity of the Hessian on an element $\xi \in H$ such that $\xi_k = 0$ if $k \notin \{i;j\}$. Then for any $t > 0, x \in H$,
\be
\phi_{ii}(t,x)\xi_i^2 + \phi_{jj}(t,x)\xi_j^2 + 2 \phi_{ij}(t,x)\xi_i \xi_j \geq 0.
\ee
Hence the result follows from the bound on $\phi_{ii}$ and $\phi_{jj}$.
\end{proof} \ca

We now establish the extension of a classical estimate on the gradient of a convex function by its second order derivatives in finite dimension.

\begin{Prop}\label{prop:grad}
Assume $\phi \in \mathcal{B}$ is a smooth solution of \eqref{hjb} with convex initial data. Assume also that $\phi_0$ is bounded from below by some constant. \ca Then we have the following estimates on the first order derivatives of the solution, for any $t > 0, x \in H$.
\begin{equation}
|\nabla \phi(t,x)| \leq \sqrt{2}\frac{\sqrt{\phi(x) - \inf_H \phi_0}}{\sqrt{t}},
\end{equation}
\be
|\phi_i(t,x)| \leq \sqrt{2}\frac{\sqrt{\phi(x) - \inf_H \phi_0}}{\sqrt{\lambda_i t}},
\ee
\be
\langle A \nabla\phi(t,x),\nabla \phi(t,x)\rangle \leq 4\left(\frac{\phi(x) - \inf_H \phi_0}{t}\right)^{2}.
\ee
\ca
\end{Prop}
This type of inequalities is more or less standard. We include a proof for the sake of completeness.
\begin{proof}
\ca We first prove the first estimate. \ca Let us consider $t > 0, x \in H$ and consider the function $ \psi : \mathbb{R} \to \mathbb{R}$ defined by
\begin{equation}
\psi(\theta) = \phi(t,x + \theta \nabla \phi(t,x)).
\end{equation}
Let us observe that $\psi$ is a smooth function and that
\begin{equation}
\psi'(\theta) = \langle \nabla \phi(t,x + \theta \nabla \phi(t,x)) , \nabla \phi(t,x) \rangle,
\end{equation}
\begin{equation}
\psi''(\theta) = \langle \nabla \phi(t,x) D^2 \phi(t,x + \theta \nabla \phi(t,x)) , \nabla \phi(t,x) \rangle.
\end{equation}
For any $\theta$, the following holds
\begin{equation}
\psi(\theta) = \psi(0) + \psi'(0) \theta + \frac{\theta(\theta - z)}{2}\psi''(z),
\end{equation}
for some $z \in [0,\theta]$. Hence
\begin{equation}
0\leq (\psi(0) - \inf_{H} \psi )  + \psi'(0) \theta + \frac{\theta^2 }{2}\|\psi''\|_{\infty}.
\end{equation}
This second order expression in $\theta$ does not change sign, thus
\begin{equation}\label{eq:551}
(\psi'(0))^2 \leq 2 \|\psi''\|_{\infty}(\psi(0) - \inf_{H} \psi ).
\end{equation}
\ca From Proposition \ref{prop:estapriori}, we deduce that
\be
\|\psi''\|_{\infty} \leq t^{-1} \langle A^{-1}\nabla \phi(t,x),\nabla \phi(t,x)\rangle.
\ee
The previous is less or equal than $t^{-1}|\nabla \phi(t,x)|^2$. Hence, we deduce from \eqref{eq:551} that 
\be
|\nabla \phi(t,x)|^2 \leq \frac{2(\phi(t,x) - \inf_H \phi)}{t}.
\ee
Hence the first inequality follows. The other two can be proven in exactly the same fashion by considering $\psi(\theta) = \phi(t,x + \theta\phi_i(t,x))$ and $\psi(\theta) = \phi(t,x + \theta A^{\frac 12}\nabla \phi(t,x))$, and by using the first estimate to prove the third one. 
\ca

\end{proof}

\subsection{A formal change of variable}
We present a change of variable which simplifies the PDE \eqref{hjb}. Formally if $\phi$ is a solution of \eqref{hjb} then defining $v$ by 
\begin{equation}\label{changev}
v(t,Bx) = \phi(t,x),
\end{equation}
where $B = A^{-1/2}$, we observe that formally $v$ satisfies
\begin{equation}\label{eqv}
\partial_t v + \frac{1}{2} |\nabla v |^2 - Tr(B^2D^2v) = 0, \text{ in } (0,T)\times B(H).
\end{equation}
\ca Let us remark that this equation falls into the classical literature on infinite dimensional second order equations, see for instance \citep{fabbri2017stochastic} which treats exactly this kind of equation. For instance, we can remark that thanks to \eqref{hypA}, the second order term is well defined for any smooth function $v$. However, when this change of variable has been made, it is by no mean clear what becomes of the initial condition, as in general, such a change of variable will send $\phi_0$ to an ill-defined function on $H$. Hence, this equation is classical for $t > 0$ but singular at $t = 0$. Let us insist on the fact that, in general, this change of variable only defines $v$, and thus the equation, on $B(H)$ on not on $H$. We shall come back on this difficulty later on.\ca

\section{Existence and uniqueness of solutions}
\subsection{Existence of a solution}For any $N \geq 0$, let us consider the Hilbert space $H_N = Span(\{e_1,...,e_N\})$. By construction $H_N$ is \ca invariant under $A$ \ca. The $N$ dimensional problem associated to \eqref{hjb} is
\begin{equation}\label{hjbN}
\partial_t \phi - \Delta_N \phi + \frac{1}{2}\langle A \nabla \phi, \nabla \phi \rangle = 0 \text{ in } (0,\infty)\times H_N,
\end{equation}
where we used the notation $\Delta_N$ to insist on the fact that we are here in $H_N$. The unknown is $\phi : [0,\infty)\times H_N \to \mathbb{R}$. The idea we are going to follow in this section is that, with an appropriate choice of initial conditions, the sequence of solutions of \eqref{hjbN} converges toward a solution of \eqref{hjb}, in a sense to be made precise later on.\\

We start with recalling the following result for problem \eqref{hjbN}.

\begin{Theorem}\label{thm:finite}
Assume $\phi_0^N \in \mathcal{C}^{1,1}(H_N)$ is convex and bounded from below, then there exists a unique viscosity solution $\phi^N$ of \eqref{hjbN} with initial condition $\phi_0^N$ which satisfies the conclusions of Propositions \ref{prop:estapriori} and \ref{prop:grad}. \end{Theorem}
\ca
The previous result belongs in some sense to a gray area of the literature on quasi-linear parabolic equations, for which we refer to \citep{lady} for a more complete presentation. Indeed, it should not appear too surprising to specialists, however, we were not able to find it stated somewhere, mainly because of the fact the initial condition is not bounded from above. We thus sketch the main argument to prove this result, while leaving the details of such a proof to the interested reader.\\
\\
\textit{Sketch of the proof of Theorem \ref{thm:finite}.} As we already said, the main difficulty in proving this result lies in the fact that the initial condition is not bounded from above. This difficulty can be overcome by realizing that from the assumption we made on the initial condition, bounding the solution of \eqref{hjbN} by using the solution of the associated deterministic problem can be done in the same way as shown in Proposition \ref{prop:est0}. Note that, in general, this estimate depends only on $\phi_0^N$, namely in its regularity and growth. This kind of technique yields an upper-bound which is sufficient to apply usual fixed point techniques to \eqref{hjbN}.
\qed

\ca
We now use in some sense the a priori estimates we presented earlier, or rather their proof as we are going to use them at the level of $H_N$. Before doing so, let us indicate the choice of the initial condition we use for \eqref{hjbN}. Several choices are suitable, and we use here the simplest one. We set for $x \in H_N$
\be
 \phi_0^N(x) = \phi_0((x,0))
\ee 
where $(x,y)$ is the vector whose first $N+1$ coordinates are given by $x$ and the other coordinates by $y$.
\ca In order to use Theorem \ref{thm:finite}, we have to ensure that the initial condition constructed in \eqref{initialN} indeed satisfies the required properties. In order to do so, we make in a first time the following assumption, which we shall relax partially in a second time.
\begin{hyp}\label{hypphi0}
The function $\phi_0 : H \to \R$ is $\mathcal{C}^{1,1}$, convex, bounded from below and satisfies 
\be\label{initialN}
\phi_0(x) \to \infty \text{ as } |x|\to \infty.
\ee
\end{hyp}
The following is immediate.
\begin{Prop}
Assume that Hypothesis \ref{hypphi0} holds, then for all $N > 1$, $\phi^N_0$ is convex, bounded from below, $\mathcal{C}^{1,1}$ and satisfies $\phi_0^N(x) \to \infty$ as $|x|\to \infty$. The $\mathcal{C}^{1,1}$ norm of $\phi_0^N$ is bounded by a constant depending only on the $\mathcal{C}^{1,1}$ norm of $\phi_0$.
\end{Prop}

\ca
 We define the solution $\psi^N$ of the deterministic problem in the finite dimensional case, which is given for $t > 0, x \in H_N$ by 
\begin{equation}
\psi^N(t,x) = \inf_{y \in H_N} \left\{ \phi_0^N(y) + \frac{\langle A^{-1}(x-y), x-y \rangle}{2t}\right\}.
\end{equation}

\ca These approximations converge in a certain sense toward respectively the initial condition and the solution of the deterministic problem. Let us introduce the lifting $\hat{\phi}_0^N((x,x')) = \phi_0^N(x)$ and $\hat{\psi}^N((x,x')) = \psi^N(x)$ for $x \in H_N, x' \in H_N^{\perp}$. The following holds.
\begin{Prop}
Assume that $\phi_0$ satisfies Hypothesis \ref{hypphi0}, then $(\hat{\phi}_0^N)_{N \geq 0}$ and $(\psi^N)_{N \geq 0}$ converge pointwise toward respectively $\phi_0$ and $\psi$. 
\end{Prop}
\begin{proof}
Since $\phi_0$ is continuous, the convergence for $\phi_0^N$ is immediate. Now, take $t > 0, x \in H$. For $N > 0$, denote by $x^N$ the orthogonal projection of $x$ onto $H_N$. Denote by $y$ the unique element of $H$ such that
\be
\psi(t,x) = \phi_0(y) + \frac{\langle A^{-1}(x-y),x-y\rangle}{2t}.
\ee
Denote by $y^N$ the orthogonal projection of $y$ onto $H_N$. By construction, for all $N> 0$,
\be
\psi^N(t,x^N) \leq \phi_0((y^N,0)) + \frac{\langle A^{-1}(x^N-y^N),x^N-y^N\rangle}{2t} \underset{N \to \infty}{\longrightarrow} \psi(t,x).
\ee
On the other hand, still by construction, 
\be
\psi(t,x) \leq \psi^N(t,x^N) + \frac{1}{2t}\sum_{i > N}\lambda_i^{-1}(x_i)^2.
\ee
Hence, we obtain that $\psi^N(t,x^N) \to \psi(t,x)$ as $N \to \infty$.

\end{proof}

Therefore, to obtain results on the limit equation \eqref{hjb}, we can pass to the limit $N \to \infty$ if we are able to construct a suitable solution for any $N < \infty$. At a fixed $N \geq 1$, we can establish the following result on the finite dimensional approximation of \eqref{hjb}. \ca
\begin{Theorem}\label{propN}
Assume that $\phi_0$ satisfies Hypothesis \ref{hypphi0}. For all $N \geq 0$, let $\phi^N$ be the solution of \eqref{hjbN} with initial condition $\phi_0^N$ given by \eqref{initialN}. Then for all $N \geq 0$, $\phi^N$ satisfies, for some continuous $\omega : \R_+ \to \R_+$, depending only on $\phi_0$ such that $\omega(0)=0$, for all $x,\xi \in H_N, t > 0$,
\be \forall 0\leq i,j \leq N, |(\phi^N)_{ij}(t,x)| \leq \frac{1}{\sqrt{\lambda_i\lambda_j}t},\ee
\be\xi D^2\phi^N(t,x)\xi \leq t^{-1}\langle A^{-1}\xi,\xi\rangle ,\ee
 \be|\nabla \phi^N(t,x)| \leq \frac{\sqrt{2(\phi^N(x) - \inf_H \phi^N_0)}}{\sqrt{t}},\ee
\be \langle A \nabla \phi^N(t,x),\nabla \phi^N(t,x)\rangle \leq 4 \left( \frac{\phi^N(x)- \inf_{H_N}\phi^N_0}{t}\right)^2,\ee
\be \label{eq654}\psi^N(t,x) \leq \phi^N(t,x) \leq \psi^N(t,x) + \omega(t).\ee
Furthermore, if we define $v(t,Bx) = \phi^N(t,x)$ then $v$ is a solution of 
\begin{equation}
\partial_t v + \frac{1}{2} |\nabla v |^2 - \sum_{i = 0}^N \lambda_i^{-1} v_{ii} = 0, \text{ in } (0,T)\times H_N
\end{equation}

\end{Theorem}
\begin{proof}
\ca Since $\phi_0^N$ satisfies all the requirements we made on $\phi_0$ in the previous section, we want to say that all the computations we made in Propositions \ref{prop:estapriori} and \ref{prop:est0} are still valid. The only points we need to take care of are the fact that $\phi^N$ is not necessary in $\mathcal{B}$, mainly because it is not $\mathcal{C}^4$ in space, and we have to ensure that the hypotheses of Proposition \ref{prop:est0} hold.\ca

The first point can be treated easily by approximation of the initial condition with $\mathcal{C}^4$ functions that we do not present here as they are standard.

Concerning the second point, let us recall that since $\phi_0^N \in \mathcal{C}^{1,1}$, \ca$\Delta \phi^N$ is uniformly bounded, possibly by a constant depending on $N$. Hence, from classical comparison principles, we can bound $\|\psi^N-\phi^N\|_{\infty}$. Thus,\ca we can apply Proposition \ref{prop:est0} and  we know that $\psi^N \leq \phi^N \leq \psi^N + \omega^N(t)$ for some modulus of continuity $\omega^N$. \ca Thanks to Remark \ref{Rem:l1}, we are lead to remark that
\be
\omega^N(t) \leq \int_0^t \sum_{i =0}^N \frac{C}{1 + \lambda_i C s}ds,
\ee
where $C$ is constant which bounds the $\mathcal{C}^{1,1}$ norm of $\phi_0$. Hence, because \eqref{hypA} is satisfied, we deduce that

\be
\omega^N(t) \leq \omega(t):= \sum_{i \geq 0} \lambda_i^{-1}\log(1 + C \lambda_i t) \underset{t \to 0}{\longrightarrow} 0.
\ee
\end{proof} 
\ca
The previous result yields some compactness on the sequence $(\phi^N)_{N \geq 1}$ of solution of \eqref{hjbN}. We now explain how it is sufficient to consider limit points of this sequence. For this, let us first define for any $N \geq 1$, $\hat{\phi}^N(t,(x,x')) = \phi^N(t,x)$ for $(x,x')\in H$ and $x \in H_N$.
\begin{Prop}\label{prop:compact}
Under the assumptions of the previous result, extracting a subsequence if necessary, the sequence $(\hat{\phi}^N)_{N \geq 1}$ converges uniformly on all compact sets toward some function $\phi \in \mathcal{C}([0,\infty),\mathcal{C}(H))$, which satisfies $\phi|_{t = 0}= \phi_0$, together with for all $t > 0, x, \xi \in H$
\be \forall 0\leq i,j , |\phi_{ij}(t,x)| \leq \frac{1}{\sqrt{\lambda_i\lambda_j}t},\ee
\be\xi D^2\phi(t,x)\xi \leq t^{-1}\langle A^{-1}\xi,\xi\rangle ,\ee
 \be|\nabla \phi(t,x)| \leq \frac{\sqrt{2(\phi(x) - \inf_H \phi_0)}}{\sqrt{t}},\ee
\be \langle A \nabla \phi(t,x),\nabla \phi(t,x)\rangle \leq 4 \left( \frac{\phi(x)- \inf_{H}\phi_0}{t}\right)^2,\ee
\be\label{eq684} \psi(t,x) \leq \phi(t,x) \leq \psi(t,x) + \omega(t).\ee
where $\omega: \R_+\to \mathbb{R}_+$ is a non-decreasing continuous function such that $\omega(0) = 0$ which depends only on $\phi_0$. 
\end{Prop} 
\begin{proof}
\ca By construction, for any $N \geq 1$, $\hat{\phi}^N$ is a solution of \eqref{hjb} with initial condition $\hat{\phi}^N_0$. We want to show that it converges in some sense toward a function $\phi : \mathbb{R}_+ \times H$. Recall that for all $N \geq 1$, $\hat{\phi}^N$ satisfies the required properties. The spatial regularity is somehow sufficient to gain compactness, given that we are able to obtain a uniform estimate on the time regularity of the sequence $(\hat{\phi}^N)_{N \geq 1}$.

The time regularity is obtained through the PDE satisfied by $\hat{\phi}^N$. Indeed, it holds that
\be\label{eq582}
|\Delta \hat{\phi}^N |+ \langle A \nabla \hat{\phi}^N,\nabla \hat{\phi}^N\rangle \leq \frac{\sum_i \lambda_i^{-1}}{t} + 4 \left( \frac{\phi^N(x)- \inf_{H_N}\phi^N_0}{t}\right)^2.
\ee
Hence, we obtain form the equation \eqref{hjbN} that the right hand side of the previous inequality bounds $|\partial_t \hat{\phi}^N|$. Thanks to the upper bound on $\hat{\phi}^N$ obtained through the value of the deterministic problem, we deduce some uniform in time regularity on all bounded sets of $(0,\infty)\times H$.\\

To make the previous heuristic more precise, let us consider an increasing sequence $(K_n)_{n\geq 0}$ of compact sets which covers $(0,\infty)\times H$. For any $n \geq 0$, the sequence $(\hat{\phi}^N)_{N \geq 0}$ is uniformly continuous on $K_n$ since it satisfies \eqref{eq582}. Uniform bounds hold from the inequality \eqref{eq654}. Hence, $(\hat{\phi}^N)_{N \geq 0}$ is uniformly continuous on each of the $K_n$. Hence, by a standard diagonal extraction, we can consider a sequence $(\hat{\phi}^k)_{k \geq 0}$ which converges uniformly on each of the $K_n$ toward a function $\phi$. The fact that the convergence is indeed uniform on all compact sets can be checked afterwards. It remains to show that $\phi$ satisfies the requirements of the result.\\

The fact that $\phi$ satisfies the regularity requirements of the problem is obtained simply by passing to the limit $k \to \infty$ on any of the compact sets $K_n$. Moreover, since for all $t> 0, x \in H$, \eqref{eq654} holds, we deduce by passing to the (pointwise) limit that \eqref{eq684} holds.
\end{proof}
\ca
\subsection{Characterization of the limit} It now remains to show in which sense the function $\phi$ given by the previous Proposition is a solution of \eqref{hjb}. The two main ingredients we use in our formulation of this fact are the notion of viscosity solution and the change of variable that we mentioned above. Namely, taking $\phi$ a function given by Proposition \ref{prop:compact}, we define a function 
\be
v(t,Bx) = \phi(t,x),
\ee
for $t > 0$ and $x \in H$. Recall that $B = A^{-\frac 12}$. A priori, $v$ is only defined on $(0,\infty)\times B(H)$, but we can prove the following result.
\begin{Lemma}\label{lemma:first}
The previous function $v$ can be extended continuously to a function $v: (0,\infty)\times H \to \R$.
\end{Lemma}
This Lemma is in fact a consequence of the following result.
\begin{Lemma}\label{lemma:second}
There exists a non-increasing non-negative function $\omega: (0,\infty)\to \R_+$ such that for any $t > 0, x \in H$ such that
\be
|\partial_t \phi(t,x) | \leq \omega(t)(|Bx|^2+1).
\ee
\end{Lemma}
\begin{proof}
We show that such an estimate is valid for any $\phi^N$ for $N \geq 1$ in the construction of $\phi$. Take $t_0> 0$. Since $\Delta \phi^N(t_0,\cdot)$ is a bounded function, uniformly in $N$, we deduce from Proposition \ref{prop:est0} that there exists $c > 0$ such that 
\be
g(t-t_0,x)\leq \phi^N(t,x) \leq g(t-t_0,x) + c(t-t_0) \text{ in } [t_0,\infty) \times H,
\ee
where $g(s,x)$ is defined by
\be
g(s,x) := \inf_{y \in H} \bigg\{\phi^N(t_0,y) + \frac{ \langle A^{-1}(x-y), x-y\rangle\rangle}{2s}\bigg\}.
\ee
Hence, an estimate on $\partial_t \phi^N(t_0,\cdot)$ reduces to proving that we have an estimate on $\partial_s g(0,\cdot)$, or in other words, on the time regularity of the deterministic problem given that the initial condition is well behaved enough. 

The minimization problem involved in the definition of $g$ admits a unique minimizer which we denote by $y(s,x)$ ($\phi^N(t_0,\cdot)$ is indeed a convex, coercive and continuous function). By construction, at this point of minimum, the following holds
\be
\nabla_x\phi^N(t_0,y(s,x)) = \frac 1 s A^{-1}(x-y(s,x)).
\ee
Moreover, we can compute the time derivative of $g$ to remark that it is given for $s > 0$ by
\be
\partial_sg(s,x) = -\frac{1}{s^2}\langle A^{-1}(x-y(s,x)),x-y(s,x)\rangle.
\ee
Hence, by definition of $y(s,x)$ we obtain that
\be
\partial_s g(s,x) = -\langle A \nabla_x \phi^N(t_0,y(s,x)),\nabla_x \phi^N(t_0,y(s,x))\rangle. 
\ee
Using the estimate on the gradient in Proposition \ref{prop:grad}, we deduce that there exists $C$ depending only on $t_0$ such that
\be
|\partial_s g(s,x) | \leq C \sqrt{\phi^N(t_0,y(s,x))} \leq C \sqrt{\phi^N(t_0,x)},
\ee
where the last inequality holds because $y(s,x)$ is constructed as an infimum. Since this bound does not depend on $s$, we deduce that it is also valid for $s=0$. Hence we finally obtain that
\be
|\partial_t \phi^N(t_0,x) | \leq C \sqrt{\phi^N(t_0,x)} + c,
\ee
from which the result easily follows recalling the estimate involving $\psi^N$ as well as the bound \eqref{boundeter}.

\end{proof}

\begin{proof}\textit{(Of Lemma \ref{lemma:first}.)}
From Lemma \ref{lemma:second}, we obtain that for $v: (0,\infty) \times H \to \R$, $\partial_tv$ is bounded by a constant which depends only on $|x|$. Recalling Proposition \ref{prop:grad}, we deduce that we have the same bound on $|\nabla_x v|^2$ since we have the relation
\be
|\nabla_x v(t,Bx)|^2 = \langle A \nabla_x \phi(t,x),\nabla_x \phi(t,x)\rangle.
\ee
Hence the result follows.
\end{proof}

Now that we have justified the existence of the function $v$, we can formulate in which sense it is a viscosity solution of 
\be\label{hjbv2}
\partial_t v + |\nabla_x v|^2 - Tr(B^2D^2v) = 0 \text{ in } (0,\infty)\times H.
\ee
In order to do so, we use finite dimensional approximation of this equation. This choice is particularly justified because of the regularity estimates available on $v$. Indeed, from Proposition \ref{prop:est0}, we obtain that for any $t> 0, x \in H$,
\be
0 \leq v_{ii}(t,x) \leq t^{-1}.
\ee
Hence, we deduce that there exists $(\epsilon_d)_{d \geq 0}$ such that $\epsilon_d \to 0$ as $d \to \infty$ and
\be\label{epsilond}
0 \leq \sum_{i > d}\lambda_{i}^{-1}v_{ii}(t,x) \leq t^{-1}\epsilon_d.
\ee
In particular, we have a good control on the rest of the second order terms if we decide to look only at the derivatives in directions $i \leq d$ in \eqref{hjbv2}. We still have to be careful with the terms involving the rest of the first order derivatives, in particular we are going to treat them differently to express the viscosity sub-solution property and a viscosity super-solution property. We can establish the following result.

\begin{Prop}\label{prop:super}
For any $d \geq 1$ and $\delta > 0$, define $v^d : (0,\infty)\times H_d \to \R$ by $v^d(t,x) = \inf_{x' \in H_d^{\perp}}\{v(t,(x,x')) + \delta |x'|^2\}.$ Then for any $(t_*,x_*) \in (0,\infty) \times H_d$ point of local minimum of $v^d- \varphi$, where $\varphi$ is a $\mathcal{C}^2$ function on $[0,\infty)\times H_d$, the following holds
\be
\partial_t \varphi(t_*,x_*) + |\nabla_x \varphi(t_*,x_*)|^2 - \sum_{i = 0}^d \lambda_i^{-1}\varphi_{ii}(t_*,x_*) \geq  - \delta^2|x_*'|^2,
\ee
where $x_*'$ is the point of minimum in $\inf_{x' \in H_d^{\perp}}\{v(t_*,(x_*,x')) + \delta |x'|^2\}.$
\end{Prop}
\begin{proof}
Take $d \geq 1, \delta > 0$, and $\varphi \in C^2(\R_+\times H_d)$. Consider $(t_*,x_*)$ a point of local minimum of $v^d - \varphi$. Without loss of generality we assume that this is a point of strict minimum. Take the sequence $(\phi^N)_{N \geq 0}$ considered in the proof of Proposition \ref{prop:compact}. Define the sequence $(v^N)_{N \geq 0}$ by
\be
v^N(t,Bx) = \phi^N(t,x).
\ee
For any $N \geq 1$, we extend $v^N$ to $H$ by setting $v^N(t,(x,x')) = v^N(t,x)$. Consider now $v^{N,d}$ defined by
\be
v^{N,d}(t,x) = \inf_{x' \in H_d^{\perp}}\{v^N(t,(x,x')) + \delta |x'|^2\}.
\ee
Denote by $(t_N,x_N,x'_N)$ the point of minimum of 
\be\label{eq:688}
(t,x,x') \to v^N(t,(x,x')) + \delta |x'|^2 - \varphi(t,x).
\ee
Remark that, for $N> d$, $x' \in H_d^{\perp}$ has only $N-d$ non zero coordinates. Since $\phi^N$ is a viscosity solution of \eqref{hjbN}, we obtain that
\be
\partial_t \varphi(t_N,x_N) + |\nabla_x \varphi(t_N,x_N)|^2 - \sum_{i = 0}^d \lambda_i^{-1}\varphi_{ii}(t_N,x_N) \geq  - \delta^2|x_N'|^2.
\ee
The result is proved if we can show that $(t_N,x_N,x'_N) \to (t_*,x_*,x'_*)$ as $N \to \infty$. In order to show that this property holds, we first gather some facts about the sequence $(v^N)_{N \geq 0}$. First, for any $t > 0$, $v^N(t,\cdot)$ is convex in $(x,x')$. Secondly, from the same arguments we developed in the proof of Lemma \ref{lemma:second}, we know that the sequence is locally Lipschitz continuous in $(t,x)$, uniformly in $N$. Hence, on all compact sets, it converges uniformly toward $v$.

From this we deduce that $(t_N,x_N) \to (t_*,x_*)$ as $N \to \infty$. The convergence of $(x'_N)_{N \geq 0}$ is a bit more subtle since it has to hold in an infinite dimensional space. The main argument to establish this relies on the addition of the term in $\delta$. Indeed, the sequence of function defined in \eqref{eq:688} is uniformly strongly convex in $x'$. Thus, the points of minimum in $x'$ converge toward $x'_N$, which proves the claim.
\end{proof}

\begin{Rem}
The previous result states that $v^d$, the finite dimensional projection of $v$, is almost a viscosity super-solution of the finite dimensional equation HJB equation. We now establish a similar result for the viscosity sub-solution property, which is simpler to state.
\end{Rem}

\begin{Prop}\label{prop:sub}
Take $d > 1$, and $\varphi \in \mathcal{C}^2(\R_+\times H_d)$. Define $v_d:(0,\infty)\times H_d \to \R$ by $v_d(t,x) = v(t,(x,0))$. Consider $(t_*,x_*) \in (0,\infty)\times H_d$, a point of local maximum of $v_d - \varphi$. Then the following holds
\be
\partial_t \varphi(t_*,x_*) + |\nabla_x \varphi(t_*,x_*)|^2 - \sum_{i = 0}^d \lambda_i^{-1}\varphi_{ii}(t_*,x_*) \leq \frac{\epsilon_d}{t_*}
\ee
\end{Prop}
\begin{proof}
The proof of this result is much simpler than the previous one. Take $d > 1, \varphi \in \mathcal{C}^2(\R_+\times H_d)$ and $(t_*,x_*) \in (0,\infty)\times H_d$, a point of local maximum of $v_d - \varphi$. Once again, we can assume without loss of generality that this is a strict maximum. We consider the sequence of functions $(v^N)_{N \geq 0}$ as in the proof of Proposition \ref{prop:super}, and consider $(t_N,x_N)$ a point of maximum of 
\be
(t,x) \to v^N(t,(x,0)) - \varphi(t,x).
\ee
In particular, $\partial_t v^N(t_N,(x_N,0)) = \partial_t \varphi(t_N,x_N)$ and $v^N_i(t_N,(x_N,0)) = \varphi_i(t_N,x_N)$ for $i\leq N$. Since $\phi^N$ is a classical solution of \eqref{hjbN}, we deduce that
\be
\begin{aligned}
\partial_t \varphi(t_N,x_N) + |\varphi(t_N,x_N)|^2  - \sum_{i = 0}^d \lambda_i^{-1}\varphi_{ii}(t_N,x_N)& = \sum_{i > d}( \lambda^{-1}_i v^N_{ii}(t_N,(x_N,0)) - v^N_i(t_N,(x_N,0))^2)\\
& \leq \frac{\epsilon_d}{t_N},
\end{aligned}
\ee
where $(\epsilon_d)_{d > 1}$ is the sequence considered in \eqref{epsilond}. Since $(t_N,x_N)\to (t_*,x_*)$ as $N \to \infty$, we obtain the required result by passing to the limit in the previous inequality.
\end{proof}
We now summarize the previous properties in the following result.

\begin{Theorem}\label{thm:existence}
Assume that $\phi_0$ satisfies Hypothesis \ref{hypphi0}. Then there exists a function $\phi \in \mathcal{C}([0,\infty),\mathcal{C}(H))$ which satisfies
\begin{itemize}
\item For all $t> 0, x \in H$
\be
\psi(t,x) \leq \phi(t,x) \leq \psi(t,x) + \omega(t)
\ee
where $\psi$ is the solution of the deterministic problem and $\omega : \R_+ \to \R_+$ is continuous, depends only on $\phi_0$ and $\omega(0) = 0$. 
\item The function $v$ defined by \eqref{changev} is uniformly Lipschitz continuous on all compact sets of $(0,\infty)\times H$ and satisfies the conclusion of Propositions \ref{prop:super} and \ref{prop:sub}.
\end{itemize}
\end{Theorem}

\subsection{Extension to more general initial conditions}
We now explain how to the previous result can be extended to more general initial conditions. Namely, we are interested in removing the $\mathcal{C}^{1,1}$ assumption. 
\begin{Theorem}\label{thm:extension}
Assume that $\phi_0$ is convex and satisfies $\phi_0(x) \to \infty$ when $|x|\to \infty$. Assume also that it can be approximated uniformly by $\mathcal{C}^{1,1}$ convex functions. Then there exists a function $\phi \in \mathcal{C}([0,\infty),\mathcal{C}(H))$ which satisfies
\begin{itemize}
\item For all $t> 0, x \in H$
\be
\psi(t,x) \leq \phi(t,x) \leq \psi(t,x) + \omega(t)
\ee
where $\psi$ is the solution of the deterministic problem and $\omega : \R_+ \to \R_+$ is continuous, depends only on $\phi_0$ and $\omega(0) = 0$. 
\item The function $v$ defined by \eqref{changev} is uniformly Lipschitz continuous on all compact sets of $(0,\infty)\times H$ and satisfies the conclusion of Propositions \ref{prop:super} and \ref{prop:sub}.
\end{itemize}
\end{Theorem}
\begin{proof}
For any $\epsilon> 0$, consider $\phi_0^{\epsilon}$ a convex $\mathcal{C}^{1,1}$ function such that $\|\phi_0 - \phi_0^\epsilon\|_{\infty} \leq \epsilon$. For all $\epsilon> 0$, $\phi_0^{\epsilon}$ satisfies Hypothesis \eqref{hypphi0}. Hence, from Theorem \ref{thm:existence}, there exists a function $\phi_{\epsilon}$ and a modulus of continuity $\omega_{\epsilon}$ satisfying the conclusions of Theorem \ref{thm:existence}.\\

The argument we want to make now is that taking the limit $\epsilon\to 0$, the function $\phi_{\epsilon}$ converges toward some function $\phi$ which satisfies the requirements. As all the a priori estimates of Theorem \ref{propN} are uniform in $\epsilon$, we focus on proving a uniform upper bound involving $\psi$, the solution of the deterministic problem with initial condition $\phi_0$. We denote by $\psi_{\epsilon}$ the solution of the deterministic problem with initial condition $\phi_0^{\epsilon}$. For any $\epsilon> 0$, it holds that for all $t \geq 0,x \in H$
\be
\begin{aligned}
 \phi_{\epsilon}(t,x)& \leq \psi_{\epsilon}(t,x) + \omega_{\epsilon}(t)\\
 & \leq \psi(t,x) + \omega_{\epsilon}(t) + \epsilon.
 \end{aligned}
\ee
Defining $\omega(t) = \inf_{\epsilon} \omega_{\epsilon}(t) + \epsilon$, we finally obtain the required estimate.
\end{proof}

\subsection{Uniqueness}
We can provide the following result of uniqueness.
\begin{Theorem}\label{thm:unique}
Under the assumptions of Theorem \ref{thm:extension}, there exists a unique function $\phi \in\mathcal{C}([0,\infty),\mathcal{C}(H))$ which satisfies the conclusions of Theorem \ref{thm:extension}.
\end{Theorem}
\begin{Rem}
The following proof is based on comparison principles for the finite dimensional equation \eqref{hjbN}. As the sub and super-solution considered are not bounded, this kind of results are not exactly the one found in usual references on viscosity solutions such as in Crandall et al. \citep{crandall1992user} or Fleming and Soner \citep{fleming2006controlled}. They rely on other techniques which despite not being as famous as the usual one, are nonetheless classical. We refer for instance to Ishii \citep{ishii1984uniqueness} and Alvarez \citep{alvarez1997bounded} for first order equations, or to Alvarez \citep{alvarez1996quasi} for second order problems.
\end{Rem}
\begin{proof}\textit{(Of Theorem \ref{thm:unique}.)}
Assume first that $\phi_0$ satisfies Hypothesis \ref{hypphi0}. Consider $\phi$ the solution constructed in the previous section as the limit of the solution $\phi^N$ of the finite dimensional problems and $\tilde{\phi}$ another solution. We denote by respectively $v$ and $\tilde{v}$ the associated functions given by the change of variable. Denote also $v^N(t,Bx) = \phi^N(t,x)$.

Take $N \geq 1$ and define the function $\tilde{v}_N (t,x) := \tilde{v}(t,(x,0))$ for $t \geq 0, x \in H_N$. We want to apply a comparison result on $\tilde{v}_N + \epsilon_N \log(t)$ and $v^N $. Recall that $v^N$ is a classical solution of 
\be\label{eq763}
\partial_t u + |\nabla_x u|^2 - \sum_{i = 0}^N\lambda^{-1}_i u_{ii} = 0 \text{ in } (0,\infty)\times H_N,
\ee

From Proposition \ref{prop:sub}, we can apply classical comparison principle results to deduce that, on any set of the form $[\gamma,T]\times H_N$ for $0<\gamma<T$ the supremum of 
\be
(t,x) \to \tilde{v}_N(t,x) - v^N(t,x) + \epsilon_N \log(t)
\ee
is reached at $t = \gamma$. Indeed, the localization of the point of maximum can be achieved easily since 
\be
\begin{aligned}
\tilde{v}_N(t,x) - v^N(t,x) &\leq \psi(t,(x,0)) - \psi^N(t,x) + \omega(t)\\
&\leq \omega(t).
\end{aligned}
\ee
From this uniform upper bound on the difference of the two functions, such a result is indeed classical. It is usually obtained by adding to $v^N$ the function $\beta \sqrt{|x|^2 + 1}$ and by letting $\beta$ goes to $0$. We refer to Alvarez \citep{alvarez1996quasi} and the reference therein for more involved comparison result between unbounded viscosity solutions\footnote{The paper \citep{alvarez1996quasi} is concerned with a stationary equation, but since we are only interested in localizing the point of maximum in the space variable, all the techniques apply instantly to our equation.}. Recalling that we have a uniform control of the solution near $t= 0$ through the value of the deterministic problem, we deduce in fact that the previous supremum vanishes as $N \to \infty$ (for instance by choosing $\gamma = \epsilon_n $). Hence we obtain that $v \geq \tilde{v}$.\\

Assume now that the previous inequality is strict, that is, that there exists $t_*>0,x_* \in H$ such that $v(t_*,x_*) > \tilde{v}(t_*,x_*)$. From the regularity of $\tilde{v}$ and the convergence of $(v^N)_{N \geq 0}$, we deduce in fact that there exists $T,\kappa,\bar{\delta},\rho > 0$ and $\bar{N} \geq 1$ such that for any $\delta \in (0,\bar{\delta}), \gamma \in (0,\frac{T}{2})$ and $N \geq \bar{N}$,
\be\label{prob899}
\inf \bigg\{ \tilde{v}(t,(x,x')) - v^N(t,x)  + \delta|x'|^2 + \rho t\bigg\} \leq -\kappa,
\ee
where the infimum is taken over $t\in [\gamma,T], x \in H_N, x' \in H_N^{\perp}$. Let us assume first that the minimum is reached at some point $(t^*,x^*,(x')^*)$. Suppose first that $t^*> \gamma$. Using Proposition \ref{prop:super} we obtain that
\be\label{prop903}
-\rho + \partial_t v^N(t^*,x^*) + \frac 12 |v^N(t^*,x^*)|^2 - \sum_{i = 0}^N \lambda_i^{-1}v^N_{ii}(t^*,x^*)\geq  - \delta^2|(x')^*|^2,
\ee
Since $v^N$ is classical solution of \eqref{eq763}, we deduce that
\be
2 \rho \leq \delta^2|(x')^*|^2.
\ee
Remark now that, by construction, $\delta |(x')^*|^2$ is bounded, hence taking $\delta$ as small as we want, we obtain that the right hand side of the previous inequality tends to $0$ as $\delta \to 0$, which is a contradiction. Hence, we necessarily have that $t^* = \gamma$. The value at the infimum then implies
\be
\tilde{v}(\gamma,(x^*,(x')^*)) - v^N(\gamma,x^*) \leq -\kappa.
\ee
From this, taking $N$ as large as we want, we obtain that $\|\tilde{v}(\gamma,\cdot) - v(\gamma,\cdot)\|_{\infty} \geq \kappa$, which contradicts the behaviour of $\phi$ near $t = 0$, taking $\gamma$ sufficiently small.\\

Thus we are able to prove that $v \leq \tilde{v}$ given that that we can prove the existence of a point of minimum in \eqref{prob899}. The general case is generally established through the perturbation of $v^N$ by a super-solution of the equation which is infinite outside a compact set, such as in Lemma 3.2 in \citep{alvarez1996quasi}. This Lemma is set for the case $\lambda_i = 1$ for the stationary equation. We leave to the reader the simple computation which consist in checking that such super-solution can be  constructed as in Lemma 3.2 in \citep{alvarez1996quasi} by replacing $|x|^2$ by $\sum_i \lambda_i x_i^2$.\\

To conclude this proof, it suffices to remark that, if $\phi_0$ is not $\mathcal{C}^{1,1}$ but can be approximated uniformly by $\mathcal{C}^{1,1}$ functions, the previous argument shows that the difference between $\tilde{\phi}$ and $\phi_{\epsilon}$, the solution given by Theorem \ref{thm:existence} of the problem associated to a $\mathcal{C}^{1,1}$ initial condition which is at distance $\epsilon$ of $\phi_0$, is at most of $\epsilon$. Hence the result is also proven by considering $\epsilon$ as small as necessary.

\end{proof}

\section*{Acknowledgments}
The three authors acknowledge a partial support from the Chair FDD from Institut Louis Bachelier and from the Lagrange Mathematics and Computation Research Center.
\bibliographystyle{plainnat}
\bibliography{bibremarks}

\end{document}